\documentclass[alphabetic]{amsart}

\usepackage{enumerate, longtable}
\usepackage{amsmath, amscd, amsfonts, amsthm, amssymb, latexsym, comment, stmaryrd, graphicx, dsfont, amsaddr, mathtools}
\usepackage{xcolor}
\usepackage[all]{xy}
\usepackage{fullpage}
\usepackage{tikz}
\usepackage[
        colorlinks,
        linkcolor=red,  citecolor=blue,
        backref
]{hyperref}

\usepackage{amstext} 
\usepackage{array}
\usepackage{cleveref}

\newtheorem{conj}{Conjecture} 
\newtheorem{thm}[conj]{Theorem}
\newtheorem{prop}[conj]{Proposition}

\newtheorem{lem}[conj]{Lemma}
\newcommand{\dom}{{\mathrm{dom}}}

\theoremstyle{definition}

\newtheorem{defi}[conj]{Definition} 
\newtheorem{rem}[conj]{Remark}

\title{Maximal Ideals in Commutative Rings and the Axiom of Choice}
\author{Alexei Entin}
\address{Raymond and Beverly Sackler School of Mathematical Sciences, Tel Aviv University, Tel Aviv 69978, Israel}
\email{aentin@tauex.tau.ac.il}

\begin{document}

\begin{abstract}
    It is well-known that within Zermelo-Fraenkel set theory (ZF), the Axiom of Choice (AC) implies the Maximal Ideal Theorem (MIT), namely that every nontrivial commutative ring has a maximal ideal. The converse implication MIT $\Rightarrow$ AC was first proved by Hodges, with subsequent proofs given by Banaschewski and Ern\'e.
    Here we give another derivation of MIT $\Rightarrow$ AC, aiming to make the exposition self-contained and accessible to non-experts with only introductory familiarity with commutative ring theory and naive set theory.
\end{abstract}

\maketitle

\section{Introduction}

In the present note we work in Zermelo-Fraenkel set theory (ZF) and all numbered assertions are theorems of ZF. By a ring we will always mean an associative ring with a unit and by a domain we mean a nontrivial commutative ring without zero-divisors. A foundational result in the theory of commutative rings (attributed to Krull) is the
\\ \\
{\bf Maximal Ideal Theorem (MIT).} Every nontrivial commutative ring $R$ has a maximal ideal.
\\ \\
This result is usually proved using Zorn's Lemma, which is equivalent to the
\\ \\
{\bf Axiom of Choice (AC).} Let $(A_i)_{i\in I}$ be a family of non-empty sets. There exists a function $f:I\to\bigcup_{i\in I}A_i$ such that $f(i)\in A_i$ for each $i\in I$. 
\\ \\
Thus we have AC $ \Rightarrow $ MIT. Scott \cite{Sco54} 
raised the question of whether 
the converse implication MIT $ 
\Rightarrow $ AC holds. It was 
answered affirmatively by 
Hodges \cite{Hod79}. 
Banaschewski \cite{Ban94} and 
Ern\'e \cite{Ern95} gave 
alternative proofs. All of the 
proofs are based on a similar 
principle - a combinatorial 
statement equivalent to AC is 
encoded in the existence of a maximal ideal in 
a suitable localized 
polynomial ring. Hodges's 
proof establishes (assuming 
MIT) that every tree has a 
branch \cite[Definition 9.10]{Jec03} (which implies AC). Banaschewski's proof establishes AC directly and Ern\'e's proof establishes the Teichm\"uller-Tuckey lemma \cite[p. 10]{Jec08} (which implies AC). The variant presented in this note establishes a weak form of Zorn's Lemma, which (as we will see below) is actually equivalent to the original Zorn's Lemma. The proof presented here was developed independently from \cite{Ban94, Ern95} (the author was initially only familiar with Hodges's work and learned about the work of Banaschewski and Ern\'e later). The material in the present note can be viewed as a special case of the results in \cite{Ern95} and is based on similar arguments. Nevertheless we hope that the presentation here will be more accessible and intuitive to non-experts and perhaps offer a new point of view that may eventually lead to novel results. The only background required from the reader is introductory familiarity with commutative rings, ideals, polynomial rings, localization, partial orders and Zorn's Lemma.

We now state a more precise version of MIT $ \Rightarrow $ AC that we will prove below (this statement is also proved in \cite{Ern95} and in slightly weaker forms in \cite{Hod79, Ban94}).
\begin{thm}\label{thm: main} Let $R$ be a domain and assume that for every set of variables $X$ and any multiplicative subset $S$ of the polynomial ring $R[X]$ with $0\not\in S$, the localization $S^{-1}R[X]$ has a maximal ideal. Then $\mathrm{AC}$ holds.\end{thm}

\section{Weak Zorn's Lemma}

Our proof does not establish AC directly, but instead establishes an equivalent statement which we call the \emph{Weak Zorn's Lemma}. Before stating it we review some basic definitions and (the original) Zorn's Lemma.

\begin{defi} Let $X$ be a partially-ordered set (poset). A subset $Y\subset X$ is a \emph{chain} if for any $x,y\in Y$ we have $x\leqslant y$ or $y\leqslant x$.\end{defi}
\begin{defi} Let $X$ be a poset, $Y\subset X$. An \emph{upper bound} on $Y$ is an element $x\in X$ such that $y\leqslant x$ for all $y\in Y$.\end{defi}
\begin{defi}\label{def: compatible} Let $X$ be a poset. A subset $Y\subset X$ is (upward) \emph{compatible} if any finite subset of $Y$ has an upper bound in $X$.\end{defi}
Note that every chain is compatible (since every finite subset has an element greater than the rest).

\begin{defi} Let $X$ be a poset. An element $x\in X$ is \emph{maximal} if there is no $y\in X$ such that $x<y$.\end{defi}
\noindent
{\bf Zorn's Lemma (ZL).} Let $X$ be a poset such that every chain $Y\subset X$ has an upper bound. Then $X$ has a maximal element.
\\ \\
{\bf Weak Zorn's Lemma (WZL).} Let $X$ be a poset such that every (upward) compatible subset $Y\subset X$ has an upper bound. Then $X$ has a maximal element.
\\ \\
Since every chain in a poset $X$ is compatible, we clearly have ZL $\Rightarrow$ WZL. In fact WZL is equivalent to ZL. We now show that WZL directly implies both ZL and AC.

\begin{prop} \label{prop: wzl}\begin{enumerate}\item[(i)]$\mathrm{WZL}\Rightarrow\mathrm{AC}$.\item[(ii)]$\mathrm{WZL}\Rightarrow\mathrm{ZL}$.\end{enumerate}\end{prop}

\begin{proof} {\bf (i).} Let $(A_i)_{i\in I}$ be a family of non-empty sets. Consider the set $X$ of \emph{partial choice functions} for the family, i.e. the set of functions $f$ with domain $\mathrm{dom}(f)\subset I$ and $f(i)\in A_i$ for all $i\in\mathrm{dom}(f)$. We order $X$ by inclusion, i.e. $f\leqslant g$ iff $\dom(f)\subset\dom(g)$ and $g|_{\dom(f)}=f$.

Let $Y\subset X$ be compatible. Then for any $f,g\in Y$ and $i\in\dom(f)\cap\dom(g)$ we have $f(i)=g(i)$. Consequently $\bigcup_{f\in Y}f$ is itself a partial choice function and is an upper bound on $Y$. Assuming WZL, there is a maximal element $f\in X$. If $\dom(f)\subsetneq I$ we can take $g=f\cup\{(i,a)\}$ for some $i\in I\setminus\dom(f),\,a\in A_i$ and then $f<g$, contradicting maximality. Therefore $\dom(f)=I$ and $f$ is a (full) choice function for $(A_i)_{i\in I}$. Thus AC holds.

{\bf (ii)} Let $X$ be a partial order where every chain has an upper bound. Let $C$ be the set of chains in $X$, partially ordered by inclusion. We want to apply WZL to $C$, so we check that its condition holds. If $D\subset C$ is compatible then for any $c_1,\ldots,c_n\in D$ we have that $\bigcup_{i=1}^nc_i\subset c\in C$ is a chain ($c$ is an upper bound on $\{c_1,\ldots,c_n\}$). Therefore $e=\bigcup_{c\in D}c$ is a chain (if $x,y\in e$ then $x\in c,\,y\in d$ for some $c,d\in D$ and thus $x,y\in c\cup d\subset u\in C$ are comparable, i.e. $x\leqslant y$ or $y\leqslant x$; here $u\in C$ is an upper bound for $\{c,d\}$). Therefore $e\in C$ is a chain and is therefore an upper bound on $D$. Assuming WZL, there must be a maximal element $c\in C$ (i.e. a maximal chain in $X$). By assumption $c$ has an upper bound $x$ in $X$. If $x$ is not maximal, say $x<y$ for some $y\in C$, then $c\cup\{x,y\}\supsetneq c$ is a chain, contradicting the maximality of $c$. Thus $X$ has a maximal element as required.

\end{proof}

We note that many important applications of Zorn's Lemma involve partial orders which satisfy the conditions of WZL, e.g. the standard proofs of MIT, the fact that every vector space has a basis and the Teichm\"uller-Tuckey lemma (but many other applications involve partial orders satisfying the condition of ZL but not of WZL). Thus WZL is a fairly natural statement, which could potentially be useful for deriving other equivalent forms of AC.

\section{Polynomial rings with partially ordered variables}

In light of Proposition \ref{prop: wzl}, Theorem \ref{thm: main} would follow at once from the following

\begin{prop}\label{prop: central} Let $R$ be a domain and $X$ a set of variables. Assume that for any multiplicative subset $S$ of the polynomial ring $R[X]$ the localization $S^{-1}R[X]$ has a maximal ideal. Then any partial order on $X$ satisfying the condition of WZL has a maximal element.\end{prop}

It remains to prove Proposition \ref{prop: central}, which will occupy the rest of this note. We fix an arbitrary poset $(X,\leqslant)$ and an arbitrary domain $R$. We view $X$ as a set of variables for the polynomial ring $R[X]$. We naturally view $X$ as a subset of $R[X]$. A \emph{monomial} in $R[X]$ is a product (possibly empty, equalling 1) of variables in $X$. We say that a monomial $m$ appears in $f\in R[X]$ if the coefficient of $m$ in $f$ is non-zero and we say that a variable $x$ appears in $m$ (or that $m$ contains $x$) if the exponent of $x$ in $m$ is positive.

\begin{defi} Let $x\in X$ be a variable and $f\in R[X]$ a polynomial. We say that $f$ is \emph{dominated by $x$} if every monomial appearing in $f$ contains a variable $y$ with $y\leqslant x$. We denote this by $f\leqslant x$.\end{defi}
We now make two observations (immediate from the definition), which will be used repeatedly.
\begin{enumerate}
\item[(i)] If $f,g\leqslant x$ and $p,q\in R[X]$ then $pf+qg\leqslant x$ (in other words the set of polynomials dominated by $x$ is an ideal).
\item[(ii)] If $f\leqslant x,\,g\not\leqslant x$ then $f+g\not\leqslant x$.
\end{enumerate}
\begin{defi} A polynomial $f\in R[X]$ is called \emph{small} if $f\leqslant x$ for some $x\in X$. It is called \emph{big} if no such $x$ exists.\end{defi}

\begin{lem}\label{lem: mult} The set $S$ of big polynomials in $R[X]$ is multiplicative.\end{lem}

\begin{proof} Clearly $1\in S$. Let $f,g\in S$ be big and let $x\in X$ be a variable. Write $f=f_1+f_2,\,g=g_1+g_2$, where $f_1$ (resp. $g_1$) consists of the monomials of $f$ (resp. $g$) dominated by $x$, and $f_2$ (resp. $g_2$) consists of the monomials of $f$ (resp. $g$) not dominated by $x$. Since $f,g$ are big we must have $f_2,g_2\neq 0$ and therefore $f_2g_2\not\leqslant x$ (since the monomials appearing in $f_2,g_2$ do not contain variables $\leqslant x$). Since $fg=(f_1g_1+f_1g_2+f_2g_1)+f_2g_2$ and $f_1g_1+f_1g_2+f_2g_1\leqslant x$ (by observation (i) above), we have $fg\not\leqslant x$ (by observation (ii) above). This is true for any $x\in X$, so $fg\in S$ is big and $S$ is multiplicative.
\end{proof}

\begin{defi} An ideal $I\lhd R[X]$ is called \emph{small} if every $f\in I$ is small.\end{defi}

\begin{defi} A \emph{maximal small ideal} is a small ideal $P\lhd R[X]$ which is not properly contained in another small ideal.\end{defi}

\begin{lem}\label{lem: loc} Let $A$ be a commutative ring, $S\subset A$ a multiplicative subset.
\begin{enumerate}\item[(i)] Assume that the localization $S^{-1}A$ has a maximal ideal. Then the set of ideals $I\lhd A$ disjoint from $S$ contains a maximal element (with respect to inclusion). 
\item[(ii)] Any such maximal element is a prime ideal.\end{enumerate}\end{lem}

\begin{proof} {\bf(i).} Let $M\lhd S^{-1}A$ be a maximal ideal and consider the localization homomorphism $\mathrm{loc}:A\to S^{-1}A$ given by $\mathrm{loc}(a)=\frac a1$. Then $\mathrm{loc}^{-1}(M)\lhd A$ is maximal among the ideals of $A$ disjoint from $S$ (it is disjoint from $S$ because $\mathrm{loc}(s)$ is invertible in $S^{-1}A$ for any $s\in S$; if $\mathrm{loc}^{-1}(M)\subsetneq I\lhd A$ then $M\subsetneq S^{-1}I\lhd S^{-1}A$, so $1\in S^{-1}I$ and therefore $I\cap S\neq\emptyset$).

{\bf(ii)} Let $I\lhd A$ be maximal in the set of ideals disjoint from $S$, $a,b\in A\setminus I$. By maximality we have $(I,a)\cap S\neq\emptyset$, so there exists $s\in (I,a)\cap S$ and similarly $t\in(I,b)\cap S$. Now $st\in (I,a)(I,b)\subset (I,ab)$ and we must have $ab\not\in I$, otherwise $st\in I\cap S$, contradicting the assumption that $I\cap S=\emptyset$. Thus $I$ is prime. \end{proof}

\begin{lem}\label{lem: max small} Assume that every localization of $R[X]$ by a multiplicative subset $S$ with $0\not\in S$ has a maximal ideal. Then there exists a maximal small ideal $P\lhd R[X]$.\end{lem}

\begin{proof} Apply the previous lemma to $A=R[X]$ and $S$ the set of big polynomials (which is multiplicative by Lemma \ref{lem: mult}).\end{proof}

\begin{lem} \label{lem: compatible} Let $Y\subset X$ be compatible (in the sense of Definition \ref{def: compatible}). Then the ideal $(Y)\lhd R[X]$ generated by $Y$ is small.\end{lem}

\begin{proof} Let $g=\sum_{i=1}^ng_iy_i\in (Y),\,y_i\in Y,\,g_i\in R[X]$. Since $Y$ is compatible there exists $x\in X$ such that $y_i\leqslant x$ for $1\le i\le n$. Therefore $g\leqslant x$ is small for any $g\in (Y)$ and $(Y)$ is a small ideal.\end{proof}

\begin{prop} \label{prop: main} Let $P\lhd R[X]$ be a maximal small ideal and denote $Y=P\cap X$. Then
\begin{enumerate}\item[(i)] $P=(Y)$ is generated by $Y$.
\item[(ii)] $Y$ is compatible.
\item[(iii)] $Y$ is \emph{maximal compatible}: if $Y\subset Y'\subset X$ and $Y'$ is compatible, then $Y'=Y$.
\end{enumerate}
\end{prop}

\begin{proof}
{\bf (i).} Clearly $P\supset (Y)$, so it is enough to show that $P\subset (Y)$. Let $f\in P$ and let $m$ be a monomial appearing in $f$. 
\\ \\{\bf Claim.} There exists $z\in Y$ which appears in $m$. \\ \\
Since $f\in P$ is small we have $m\neq 1$, so let $x$ be a variable appearing in $m$. Let $g\in P$ and consider $h=x^df+g\in P$, where $d>\deg g$. Since $h$ is small and the monomials appearing in $g$ and $x^dm$ also appear in $h$ because of our assumption on $d$ (they cannot cancel each other out), there exists $y\in X$ such that $h\leqslant y$ and therefore $m,g\leqslant y$ (note that since $x$ occurs in $m$ the condition $x^dm\leqslant y$ is equivalent to $m\leqslant y$). This implies that $g+qm$ is small for any $q\in R[X],\,g\in P$ and therefore $(P,m)$ is a small ideal. By the maximality of $P$ we must have $m\in P$. By Lemma \ref{lem: loc}(ii) the ideal $P$ is prime and therefore one of the variables appearing in $m$ lies in $P\cap X=Y$, establishing the claim.\\

Now since every monomial appearing in $f$ contains a variable from $Y$, we have $f\in (Y)$ and the proof of (i) is complete.

{\bf (ii).} Let $x_1,\ldots x_n\in Y$. We assume WLOG that they are distinct. Then $x_1+x_2+\ldots+x_n\in P$ is small and therefore there exists an upper bound $x\in X$ on $\{x_1,\ldots,x_n\}$. Hence $Y$ is compatible.

{\bf (iii)} Let $Y\subset Y'\subset X$ be compatible, $x\in Y'$. By Lemma \ref{lem: compatible}, $(Y')\supset (Y)=P$ is a small ideal. By the maximality of $P$ we have $(Y')=P$ and therefore $Y'=(Y')\cap X=P\cap X=Y$.
\end{proof}

\begin{rem}
 The converse of Proposition \ref{prop: main}(iii) also holds: if $Y\subset X$ is maximal compatible, the ideal $(Y)$ is a maximal small ideal in $R[X]$. Indeed, $(Y)$ is small by Lemma \ref{lem: compatible} and if $f\in R[X]\setminus (Y)$ then either $f$ is a constant and thus big, or $f$ has a monomial $m=x_1\cdots x_n$ with $x_i\not\in Y$ for $1\le i\le n$. Since $Y\cup\{x_i\}$ is not compatible for any $i$, one can pick $y_{ij}\in Y,\,1\le i\le n,\,1\le j\le n_i$ such that each $\{y_{i1},\ldots,y_{in_i},x_i\}$ has no upper bound in $X$. Therefore $f+\sum_{y\in \{y_{ij}:1\le i\le n,1\le j\le n_{ij}\}}y$ is big. In either case the ideal $(Y,f)\lhd R[X]$ is not small and therefore $(Y)$ is a maximal small ideal. Thus there is a bijection between maximal compatible subsets of $X$ and maximal small ideals of $R[X]$. This is a special case of \cite[Proposition on p. 126]{Ern95}.
\end{rem}

\section{Conclusion of the proof}
    
\begin{proof}[Proof of Proposition \ref{prop: central}] Let $X$ be a poset such that every compatible $Y\subset X$ has an upper bound. By the assumption of Proposition \ref{prop: central} and Lemma \ref{lem: max small} there exists a maximal small ideal $P\lhd R[X]$. By Proposition \ref{prop: main}(ii-iii) the set $Y=P\cap X\subset X$ is maximal compatible and by assumption $Y$ has an upper bound $x\in X$. We claim that $x$ is a maximal element of $X$. Otherwise $x<y$ for some $y\in Y$ and $Y\cup\{x,y\}\supsetneq Y$ is compatible (since $y$ is an upper bound on the entire set), contradicting the maximality of $Y$.
\end{proof}

{\bf Acknowledgment.} The author would like to thank Jorgen Harmse for his remarks on a previous version of this note, particularly for spotting a gap in the proof of Proposition \ref{prop: central}(i). The author was partially supported by Israel Science Foundation grant no. 2507/19.

\bibliography{mybib}
\bibliographystyle{alpha}

\end{document}